\documentclass[12pt, a4paper]{article}%
\usepackage{graphicx}
\usepackage{amsmath}
\usepackage{amsfonts}
\usepackage{amssymb}
\usepackage{graphicx}
\usepackage[a4paper,left=2cm,right=2cm,top=2cm,bottom=2cm,footskip=0.25cm]%
{geometry}%
\providecommand{\U}[1]{\protect\rule{.1in}{.1in}}
\newtheorem{theorem}{Theorem}

\newtheorem{lemma}[theorem]{Lemma}

\newtheorem{proposition}[theorem]{Proposition}

\newenvironment{proof}[1][Proof]{\noindent\textbf{#1.} }{\ \rule{0.5em}{0.5em}}
\begin{document}

\title{A lower bound of the energy functional of a class of vector fields and a
characterization of the sphere}
\author{Giovanni Nunes
\and Jaime Ripoll}
\date{}
\maketitle

\begin{abstract}
Let $M$ be a compact, orientable, $n-$dimensional Riemannian manifold,
$n\geq2,$ and let $F$ be the energy functional acting on the space $\Xi(M)$ of
$C^{\infty}$ vector fiels of $M$,%
\[
F(X):=\frac{\int_{M}|\left\Vert \nabla X\right\Vert ^{2}dM}{\int_{M}\left\Vert
|X\right\Vert ^{2}dM},\text{ }X\in\Xi(M)\backslash\{0\}.
\]

Let $G\subset\operatorname*{Iso}(M)$ be a compact Lie subgroup of the isometry
group of $M$ acting with cohomogeneity $1$ on $M.$ Assume that any isotropy
subgroup of $G$ is non trivial and acts with no fixed points on the tangent
spaces of $M$, except at the null vectors. We prove in this note that under
these hypothesis, if the Ricci curvature $\operatorname*{Ric}\nolimits_{M}$ of
$M$ has the lower bound $\operatorname*{Ric}\nolimits_{M}\geq\left(
n-1\right)  k^{2},$ then $F(X)\geq k^{2},$ for any $G-$invariant vector field
$X\in\Xi(M)\backslash\{0\}$, and the equality occurs if and only if $M$ is
isometric to the $n-$dimensional sphere $\mathbb{S}_{k}^{n}$ of constant
sectional curvature $k^{2}$. In this case $X$ is an infimum of $F$ on
$\Xi(\mathbb{S}_{k}^{n}).$

\end{abstract}

\section{Introduction}

\qquad Given a compact, orientable, $n-$dimensional Riemannian manifold $M,$
$n\geq2,$ the energy functional $F$ of vector fields of $M$ is defined by%
\[
F(X):=\frac{\int_{M}\left\Vert \nabla X\right\Vert ^{2}dM}{\int_{M}\left\Vert
X\right\Vert ^{2}dM},\text{ }X\in\Xi(M)\backslash\{0\},
\]
where $\nabla$ is the Riemannian connection of $M$ and $\Xi(M)$ the space of
$C^{\infty}$ vector fields of $M.$ It is well known that the critical values
of $F$ are the eingenvalues of the so called rough Laplacian
$-\operatorname*{div}\nabla$ of $M$ and form a discrete sequence $0\leq
\lambda_{0}<\lambda_{1}<...\rightarrow\infty.$

In \cite{NR} it is proved that if $M$ is the Euclidean sphere $\mathbb{S}%
_{k}^{n}$ of constant sectional curvature $k^{2},$ $n\geq2,$ then the infimum
of $F$ is assumed by a vector field $X,$ orthogonal to the orbits of a compact
Lie subgroup of the isometry group of $\mathbb{S}_{k}^{n},$ acting with
cohomogeneity $1$ in $\mathbb{S}_{k}^{n}$ and having a fixed point (the
orthogonal subgroup $\operatorname*{O}(n-1)$ of $\operatorname*{O}(n)$ fixing
a point of $\mathbb{S}_{k}^{n}).$ Moreover, $X$ is $G-$invariant. We
investigate here a possible extension of this result to a Riemannian manifold
and arrived to a theorem which gives a characterization of the spheres as the
only minimizers for $F$ on a certain space of vector fields:

\begin{theorem}
\label{main}Let $M$ be a compact, orientable, $n-$dimensional Riemannian
manifold, $n\geq2,$ and $G\subset\operatorname*{Iso}(M)$ a compact Lie
subgroup of the isometry group of $M$ acting with cohomogeneity $1$ on $M.$
Assume that any isotropy subgroup of $G$ is non trivial and acts with no fixed
points on the tangent spaces of $M$, except at the null vectors. If the Ricci
curvature $\operatorname*{Ric}\nolimits_{M}$ of $M$ has the lower bound
$\operatorname*{Ric}\nolimits_{M}\geq\left(  n-1\right)  k^{2},$ then
$F(X)\geq k^{2},$ for any $G-$invariant vector field $X\in\Xi(M)$, and the
equality occurs if and only if $M$ is isometric to the $n-$dimensional sphere
$\mathbb{S}_{k}^{n}$ of constant sectional curvature $k^{2}$. In this case $X$
is an infimum of $F$ on $\Xi(\mathbb{S}_{k}^{n}).$
\end{theorem}

We observe that rank $1$ compact symmetric spaces admit a compact Lie group of
isometries satisfying the conditions of Theorem \ref{main}. More generally,
any $G-$ symmetric compact Riemannian manifold, where $G$ is the isotropy
subgroup of a compact rank $1$ symmetric space\footnote{The definition of a
$G-$ invariant Riemannian manifold is a natural and direct generalization of a
rotationally symmetric manifold as defined in (\cite{C}, definition 3.1)} ($G$
is isomorphic to one of the groups: $\operatorname*{O}(n),$ $\operatorname*{U}%
(n)\times\operatorname*{U}(1),$ $\operatorname*{Sp}(n)\times\operatorname*{Sp}%
(1),$ $\operatorname*{Spin}(9)).$ In particular, rotationally symmetric
compact Riemannian manifolds ($G$ is isomorphic to $\operatorname*{O}(n)$).

One should also mention other non trivial examples of compact cohomogeneity
one actions. They are given by certain submanifolds of $\mathbb{R}^{n}$ using
a construction introduced in \cite{HL}, as follows: Let $G$ be a compact Lie
subgroup of the isometry group of $\mathbb{R}^{n},$ $n\geq3,$ acting with
cohomogeneity $k\geq2,$ that is, the principal orbits of $G$ are submanifolds
of $\mathbb{R}^{n}$ with codimension $k.$ It is known that there is an open
dense subset $\left(  \mathbb{R}^{n}/G\right)  ^{\ast}$ of the quotient space
$\mathbb{R}^{n}/G,$ with the quotient topology, which is a smooth manifold of
dimension $k$ (see the Theorem of Section 2 of \cite{HL}). If $\gamma$ is a
closed curve in $\left(  \mathbb{R}^{n}/G\right)  ^{\ast}$ and $\pi
:\mathbb{R}^{n}\rightarrow\mathbb{R}^{n}/G$ the quotient projection, then
$M:=\pi^{-1}\left(  \gamma\right)  $ is a compact $G-$invariant manifold of
dimension $n-k+1.$ Clearly, $G$ acts with cohomogeneity $1$ in $M.$ The
compact subgroups of $\operatorname*{Iso}\left(  \mathbb{R}^{n}\right)  $
acting with cohomogeneity $2$ and $3$ are classified in Theorems 5, 6 and 7 of
\cite{HL} and explicit examples can be constructed using the representations
of these groups given in these theorems.

\section{Preliminary and prior results}

\qquad In this section we quote some Theorems of \cite{NR} and prove other
results which we will be used in the proof of Theorem \ref{main}.

\begin{theorem}
\label{inv} Let $M$ be a compact $n-$dimensional Riemannian manifold,
$n\geq2,$ and $G\subset\operatorname*{Iso}(M)$ a compact Lie subgroup of the
isometry group of $M$. Then the set $S_{G}$ of\ critical points of $F$
restricted to the subspace of $G-$invariant vector fields of $\Xi\left(
M\right)  $ is contained in the set of critical points of $F$ on $\Xi\left(
M\right)  $.
\end{theorem}

\begin{theorem}
\label{inf}Let $\mathbb{S}_{k}\subset\mathbb{R}^{n+1},$ $k>0,$ $n\geq2,$ be
the $n-$dimensional sphere of constant sectional curvature $k^{2}.$ Then the
infimum of $F$ on $\Xi\left(  \mathbb{S}_{k}\right)  $ is $k^{2}$ and is
assumed by the orthogonal projection on $T\mathbb{S}_{k}$ of a constant
nonzero vector field of $\mathbb{R}^{n+1}$.
\end{theorem}

These two theorems above are proved, respectively, in Theorem 3, Section 3 and
Theorem 1, Section 2 of \cite{NR} . We also need the following results:

\begin{lemma}
\label{gradiente}Let $M$ be a compact $n-$dimensional Riemannian manifold,
$n\geq2.$ If $V\in\Xi(M)\backslash\{0\}$ is a gradient vector field$,$ then
\[
F\left(  V\right)  \geq\frac{1}{n-1}\int_{M}\operatorname*{Ric}\nolimits_{M}%
\left(  V,V\right)  dM.
\]

\end{lemma}

\begin{proof}
Assume that $V=\operatorname{grad}h,$ $h\in C^{\infty}\left(  M\right)
\backslash\{0\}.$ From Bochner's formula,%
\begin{align*}
\int_{M}\left(  \Delta h\right)  ^{2}dM  &  =\int_{M}\operatorname*{Ric}%
\nolimits_{M}\left(  \operatorname*{grad}h,\operatorname*{grad}h\right)
dM+\int_{M}\left\vert \operatorname*{Hess}\left(  h\right)  \right\vert
^{2}dM\\
&  =\int_{M}\operatorname*{Ric}\nolimits_{M}\left(  V,V\right)  dM+\int
_{M}\left\vert \operatorname*{Hess}\left(  h\right)  \right\vert ^{2}dM.
\end{align*}
The proof then follows by using the inequality $\left(  \Delta h\right)
^{2}\leq n\left\vert \operatorname*{Hess}\left(  h\right)  \right\vert ^{2}$
and observing that $\left\vert \operatorname*{Hess}\left(  h\right)
\right\vert =\left\Vert \nabla V\right\Vert .$
\end{proof}

\begin{lemma}
\label{operator}Let $M$ be a compact $n-$dimensional Riemannian manifold,
$n\geq2,$ and $G\subset\operatorname*{Iso}(M)$ a compact Lie subgroup of the
isometry group of $M$ acting with cohomogeneity $1$ on $M.$ If $V\in\Xi(M)$ is
a $G$-invariant vector field orthogonal to the orbits of $G$ then $V$ is the
gradient of a $C^{\infty}$ function in $M$ .
\end{lemma}

\begin{proof}
From the Theorem, Section 2 of \cite{HL}, there is an open dense submanifold
$M^{\ast}$ of $M$ such that any orbit of $G$ through a point of $M^{\ast}$ is
contained in $M^{\ast}$, has dimension $n-1$ and, hence, divides $M$ into two
connected components. Let $O\subset M^{\ast}$ be an orbit of $G$ in $M^{\ast}$
and let $d:M^{\ast}\rightarrow\mathbb{R}$ be the oriented distance to $O.$
Since $\operatorname{grad}d$ and $V$ are $G-$invariant the function
$f(p):=\left\langle V\left(  p\right)  ,\operatorname{grad}d\left(  p\right)
\right\rangle ,$ $p\in M^{\ast},$ is $G-$invariant. Moreover, since $V$ is
orthogonal to the orbits of $G$ we have $V=f\operatorname{grad}d$ at $M^{\ast
}.$ Since $V$ is $C^{\infty}$ in $M$ but $\operatorname{grad}d$ is not defined
at $M\backslash M^{\ast}$ the vector field $V$ must vanish at $M\backslash
M^{\ast}.$ It follows that $f$ extends continuously to $M$ as $0$ at
$M\backslash\partial M.$ Setting%
\[
a=\min_{M}d,\text{ \ }b=\max_{M}d,
\]
since $f$ is $G-$invariant, we may define a real function $\Phi\in
C^{0}\left(  \left[  a,b\right]  \right)  $ by%
\[
\Phi\left(  t\right)  =s\Leftrightarrow f(p)=s,\text{ }t\in\left[  a,b\right]
,
\]
where $p\in M$ is such that $t=d(p,O).$ Defining $\phi\in C^{1}\left(  \left[
a,b\right]  \right)  $ by%
\begin{equation}
\phi(t)=\int_{a}^{t}\Phi(s)ds,\text{ }t\in\left[  a,b\right]  , \label{H}%
\end{equation}
and $h\in C^{\infty}\left(  M^{\ast}\right)  \cap C^{1}\left(  M\right)  $ by
$h=\phi\circ d$ we have%
\begin{align*}
\operatorname{grad}h  &  =\left(  \phi^{\prime}\circ d\right)
\operatorname{grad}d=\left(  \Phi\circ d\right)  \operatorname{grad}%
d=f\operatorname{grad}d\\
&  =\left\langle V\left(  p\right)  ,\operatorname{grad}d\left(  p\right)
\right\rangle \operatorname{grad}d=V
\end{align*}
in $M.$ This proves that $V$ is the gradient of a the $C^{1}$ function $h$ in
$M.$ Since $V$ is $C^{\infty}$ it follows that $h\in C^{\infty}\left(
M\right)  .$ This concludes with the proof of the lemma.
\end{proof}

\begin{proposition}
\label{corolario}Let $M$ be a compact $n-$dimensional Riemannian manifold,
$n\geq2,$ and $G\subset\operatorname*{Iso}(M)$ a compact Lie subgroup of the
isometry group of $M$ acting with cohomogeneity $1$ on $M.$ Suppose that any
isotropy subgroup of $G$ is non trivial and acts with no fixed points on the
tangent spaces of $M$, except at the null vectors. Then, if $V$ is a $G-$
invariant vector field, $V$ is orthogonal to the orbits of $G$.
\end{proposition}

\begin{proof}
Let $V$ $\in\Xi\left(  M\right)  $ be $G-$invariant$.$ Given $p\in M,$ if%
\[
G(p):=\left\{  g(p)\text{
$\vert$
}g\in G\right\}  =\left\{  p\right\}
\]
then there is nothing to prove. If $G(p)\neq\left\{  p\right\}  $ then, since
\[
G_{p}:=\left\{  g\in G\text{
$\vert$
}g(p)=p\right\}  \neq\left\{  \operatorname*{Id}\right\}
\]
by hypothesis, we may take $g\in G_{p},$ $g\neq\operatorname*{Id}.$ Let $N$ be
any non zero vector field orthogonal to $G(p)$ in a neighborhood of $p$. Set
$V_{N}=\left\langle V,N\right\rangle N$ and $V_{T}=V-V_{N}.$ We have
$dg_{p}V_{N}\left(  p\right)  =V_{N}(p)$ since $dg_{p}$ acts trivially on
$\left(  T_{p}G(p)\right)  ^{\bot}$. We then obtain%
\begin{align*}
V\left(  g\left(  p\right)  \right)   &  =dg_{p}V\left(  p\right)
=dg_{p}V_{T}\left(  p\right)  +dg_{p}V_{N}\left(  p\right) \\
&  =dg_{p}V_{T}\left(  p\right)  +V_{N}\left(  p\right)  .
\end{align*}

Since
\[
V(g(p))=V(p)=V_{T}(p)+V_{N}(p)
\]
we obtain
\begin{equation}
dg_{p}V_{T}\left(  p\right)  =V_{T}\left(  p\right)  . \label{fix}%
\end{equation}
Since (\ref{fix}) holds for any $g\in G_{p}$ and $G_{p}$ has non nonzero fixed
vectors it follows that $V_{T}=0$ and $V$ is orthogonal to the orbits of $G.$
\end{proof}

\begin{proposition}
\label{laplaquad}Let $M$ be a compact $n-$dimensional Riemannian manifold,
$n\geq2,$ $G\subset\operatorname*{Iso}(M)$ a compact Lie subgroup of the
isometry group of $M$ acting with cohomogeneity $1$ on $M.$ Let $N$ be a
unitary vector field orthogonal to the principal orbits of $G.$ Assume that
$h:M\rightarrow\mathbb{R}$ is a $C^{\infty}$ and $G-$invariant function. Then
$f:=\left\langle \operatorname{grad}h,N\right\rangle \in C^{\infty}\left(
M^{\ast}\right)  $ and it holds, in $M^{\ast},$ the formula%
\begin{equation}
\left(  \Delta h\right)  ^{2}=\left(  N\left(  f\right)  \right)
^{2}-2\left(  n-1\right)  fHN\left(  f\right)  +\left(  n-1\right)  ^{2}%
f^{2}H^{2}, \label{dh}%
\end{equation}
where $M^{\ast}$ is as in Lemma \ref{operator} and $H,$ $\left\vert
B\right\vert :M^{\ast}\rightarrow\mathbb{R}$, at given $x\in M^{\ast},$ are
the mean curvature and the norm of the second fundamental form of the orbit
$G\left(  x\right)  $ with respect to $N$.
\end{proposition}

\begin{proof}
Let $x\in M^{\ast}$ and $p\in G\left(  x\right)  $ be given$.$ Let
$\{E_{1},E_{2},...,E_{n-1}\}$ be an orthonormal frame of $TG(x)$ in a
neighborhood of $p$. Then%
\begin{align*}
\Delta h  &  =\operatorname{div}\operatorname{grad}h\\
&  =\left\langle \nabla_{N}\operatorname{grad}h,N\right\rangle +\overset
{n-1}{\underset{i=1}{%
{\displaystyle\sum}
}}\left\langle \nabla_{E_{1}}\operatorname{grad}h,E_{i}\right\rangle \\
&  =\left\langle \nabla_{N}N\left(  h\right)  N,N\right\rangle +\overset
{n-1}{\underset{i=1}{%
{\displaystyle\sum}
}}\left\langle \nabla_{E_{1}}N\left(  h\right)  N,E_{i}\right\rangle \\
&  =N\left(  N\left(  h\right)  \right)  +N(h)\overset{n-1}{\underset{i=1}{%
{\displaystyle\sum}
}}\left\langle \nabla_{E_{1}}N,E_{i}\right\rangle \\
&  =N(f)-(n-1)fH
\end{align*}
which gives (\ref{dh}).
\end{proof}

\section{Proof of Theorem \ref{main}}

\qquad It is easy to see that we may assume, with no loss of generality, that
$k=1$. Let $V\in\Xi\left(  M\right)  \backslash\left\{  0\right\}  $ be a
$G-$invariant vector field$.$ By the hypothesis of the theorem and from
Proposition \ref{corolario}, it follows that $V$ is orthogonal to the orbits
of $G.$ Therefore, by Lemma \ref{operator}, $V$ is a gradient vector field and
hence, by Lemma \ref{gradiente},
\begin{equation}
F\left(  V\right)  \geq\frac{1}{n-1}\int_{M}\operatorname*{Ric}\nolimits_{M}%
\left(  V,V\right)  dM\geq1, \label{eq}%
\end{equation}
proving the first part of the theorem.

Assume now that $F(V)=1$ for some $G-$invariant vector field $V\in
\Xi(M)\backslash\left\{  0\right\}  $ and let $h\in C^{\infty}\left(
M\right)  $ such that $V=\operatorname{grad}h.$

Assuming that
\[
\int_{M}\left\Vert V\right\Vert ^{2}=1
\]
we have, by Bochner's formula,%
\[
\int_{M}\left(  \Delta h\right)  ^{2}dM=\int_{M}\operatorname*{Ric}%
\nolimits_{M}\left(  \operatorname*{grad}h,\operatorname*{grad}h\right)
dM+\int_{M}\left\vert \operatorname*{Hess}\left(  h\right)  \right\vert
^{2}dM.
\]

We then obtain%
\[
\int_{M}n\left\vert \operatorname*{Hess}\left(  h\right)  \right\vert
^{2}dM\geq\int_{M}\operatorname*{Ric}\nolimits_{M}\left(  \operatorname*{grad}%
h,\operatorname*{grad}h\right)  dM+\int_{M}\left\vert \operatorname*{Hess}%
\left(  h\right)  \right\vert ^{2}dM
\]
or%
\[
\left(  n-1\right)  \int_{M}\left\vert \operatorname*{Hess}\left(  h\right)
\right\vert ^{2}dM\geq\int_{M}\left\Vert \operatorname{grad}h\right\Vert
^{2}\operatorname*{Ric}\nolimits_{M}\left(  \frac{\operatorname*{grad}%
h}{\left\Vert \operatorname{grad}h\right\Vert },\frac{\operatorname*{grad}%
h}{\left\Vert \operatorname{grad}h\right\Vert }\right)  dM\geq\left(
n-1\right)  .
\]

Hence,%
\begin{equation}
\int_{M}\left\vert \operatorname*{Hess}\left(  h\right)  \right\vert
^{2}dM\geq1. \label{eg}%
\end{equation}

Since $V=\operatorname{grad}h,$ $\left\vert \operatorname*{Hess}\left(
h\right)  \right\vert =\left\Vert \nabla V\right\Vert ,$ it follows from
(\ref{eq}) that the equality $F\left(  V\right)  =1$ occurs if and only if
\begin{equation}
\left(  \Delta h\right)  ^{2}=n\left\vert \operatorname*{Hess}\left(
h\right)  \right\vert ^{2}, \label{laplacian}%
\end{equation}
and%
\[
\operatorname*{Ric}\nolimits_{M}\left(  \frac{\operatorname*{grad}%
h}{\left\Vert \operatorname{grad}h\right\Vert },\frac{\operatorname*{grad}%
h}{\left\Vert \operatorname{grad}h\right\Vert }\right)  =n-1.
\]

Putting $f=\left\langle \operatorname{grad}h,N\right\rangle $ we have $V=fN$
and, by Proposition \ref{laplaquad}, the following equation holds in $M^{\ast
}$%
\[
\left(  \Delta h\right)  ^{2}=\left(  N\left(  f\right)  \right)
^{2}-2\left(  n-1\right)  HfN\left(  f\right)  +\left(  n-1\right)  ^{2}%
f^{2}H^{2}.
\]

Noting that
\begin{align*}
\left\vert \operatorname*{Hess}\left(  h\right)  \right\vert ^{2} &  =\left(
N\left(  N\left(  h\right)  \right)  \right)  ^{2}+\left(  N\left(  h\right)
\right)  ^{2}\left\vert B\right\vert ^{2}\\
&  =\left(  N\left(  f\right)  \right)  ^{2}+f^{2}\left\vert B\right\vert
^{2}.
\end{align*}

we obtain that (\ref{laplacian}) is equivalent to%
\begin{equation}
\left(  N\left(  f\right)  +Hf\right)  ^{2}=f^{2}\left[  -\frac{n\left\vert
B\right\vert ^{2}}{n-1}+nH^{2}\right]  \label{f}%
\end{equation}
which implies that%
\begin{equation}
-\frac{\left\vert B\right\vert ^{2}}{n-1}+H^{2}\geq0.\label{B}%
\end{equation}

But
\[
\left[  \left(  n-1\right)  H\right]  ^{2}\leq\left(  n-1\right)  \left\vert
B\right\vert ^{2},
\]
and therefore
\begin{equation}
-\frac{\left\vert B\right\vert ^{2}}{n-1}+H^{2}=0.\label{z}%
\end{equation}
From (\ref{f}),
\begin{equation}
N\left(  f\right)  +Hf=0.\label{phi}%
\end{equation}

Since $V$ is a critical point of $F$ with eigenvalue $1$, we have
$\operatorname*{div}\nabla\left(  fN\right)  =fN.$ A calculation gives%
\[
\operatorname*{div}\nabla\left(  fN\right)  =(N(N(f))-\left(  n-1\right)
HN\left(  f\right)  -\left\vert B\right\vert ^{2}f)N
\]
so that%
\begin{equation}
N(N(f))-\left(  n-1\right)  HN\left(  f\right)  -\left\vert B\right\vert
^{2}f=-f. \label{p}%
\end{equation}

Using (\ref{phi}) we obtain%
\[
N(N(f))-\left(  n-1\right)  HN\left(  f\right)  -\left\vert B\right\vert
^{2}f=N\left(  N\left(  f\right)  \right)  +(n-1)\left[  H^{2}-\left\vert
B\right\vert ^{2}/\left(  n-1\right)  \right]  f=-f
\]
and, using (\ref{z}),
\[
N\left(  N\left(  f\right)  \right)  =-f.
\]
\ 

Now, note that
\begin{align*}
\Delta\left(  N\left(  f\right)  \right)   &  =\\
&  =\left(  n-1\right)  N\left(  N\left(  f\right)  \right)  H-N\left(
N\left(  N\left(  f\right)  \right)  \right)  =-\left(  n-1\right)
fH+N\left(  f\right) \\
&  =+\left(  n-1\right)  N\left(  f\right)  +N\left(  f\right)  =nN\left(
f\right)  .
\end{align*}

It follows that $N\left(  f\right)  $ is an eigenfunction for the usual
Laplacian in $M$ with eigenvalue $n$. As, by hypothesis, $\operatorname*{Ric}%
\nolimits_{M}\geq\left(  n-1\right)  ,$ it follows from a classic result of
Obata, that $M$ is isometric to the unit sphere $\mathbb{S}^{n}$ (\cite{O})$.$

From Theorem \ref{inf} the converse also follows, that is, if $M$ is a sphere
of radius $1$ then $F(V)=1.$

This concludes with the proof of the theorem.

\end{document}